\documentclass[11pt]{amsart}
\usepackage{amssymb}

\newtheorem{theorem}{Theorem}[section]
\newtheorem{lemma}[theorem]{Lemma}
\newtheorem{remark}[theorem]{Remark}
\newtheorem{definition}[theorem]{Definition}

\newtheorem{corollary}[theorem]{Corollary}
\newtheorem{proposition}[theorem]{Proposition}

\newtheorem{lem-def}[theorem]{Lemma-Definition}

\newcommand{\hooklongrightarrow}{\lhook\joinrel\longrightarrow}
\renewenvironment{proof}{{\bfseries Proof.}}{\qed}
\newcommand{\R}{\mathbb R}

\newcommand{\N}{\mathbb N}

\newcommand{\Q}{\mathbb Q}

\def\op{\operatorname}

\def\al{\alpha}
\def\ali{\alpha_\infty}
\def\as#1{\renewcommand\arraystretch{#1}}

\def\bs{\vskip.5cm}
\def\be{\beta}
\def\bi{b_\infty}

\def\chr{\op{char}}
\def\comb#1#2{\as{0.8}\left(\!\!\begin{array}{c}
#1\\#2
\end{array}\!\!\right)\as{1}}
\def\cl#1{[#1]_\nu}

\def\dg{\op{deg}}
\def\di{\delta_\infty}

\def\diso{\lower.4ex\hbox{$\downarrow$}\raise.4ex\hbox{\mbox{\scriptsize
$\wr$}}}
\def\dm{\Delta}

\def\e{\medskip}

\def\ep#1{\exp(\Pi i#1)}
\def\ep{\epsilon}

\def\erel{e_{\op{rel}}}

\def\g{\Gamma}
\def\ga{\gamma}

\def\gg{\mathcal{G}}

\def\ggn{\mathcal{G}_\nu}

\def\gi{\g_{\infty}}

\def\gn{\g_\nu}

\def\gq{\g_\Q}
\def\gr{\operatorname{gr}}

\def\inn{\op{in}_\nu}

\def\imp{\ \Longrightarrow\ }

\def\iso{\ \lower.3ex\hbox{\as{.08}$\begin{array}{c}\lra\\\mbox{\tiny $\sim\,$}\end{array}$}\ }

\def\ka{\kappa}

\def\km{k_\nu}
\def\kp{\op{KP}}
\def\kpcal{\mathcal{KP}}
\def\kpi{\op{KP}_\infty}
\def\kpm{\op{KP}(\mu)}
\def\kpn{\op{KP}(\nu)}

\def\kpr{\op{KP}(\rho)}

\def\kx{K[x]}

\def\La{\Lambda}

\def\lx{\operatorname{lex}}
\def\lg{l\raise.6ex\hbox to.2em{\hss.\hss}l}

\def\lra{\,\longrightarrow\,}
\def\lx{\operatorname{lex}}

\def\mi{m_\infty}

\def\ml{\op{mult}}

\def\mmu{\mid_\nu}

\def\orb{\hbox to  .3em{$\backslash$}\backslash}

\def\p{\mathfrak{p}}
\def\pb#1{\partial_b(#1)}
\def\pcv{\op{PrCvx}}

\def\phmn{\Phi_{\mu,\nu}}

\def\ppa{\mathcal{P}_{\alpha}}
\def\prt#1#2{\partial_{#1}(#2)}
\def\pset{\mathcal{P}}

\def\rhi{\rho_{\infty}}
\def\rlex{\R^I_{\lx}}

\def\sii{\ \Longleftrightarrow\ }

\def\smu{\sim_\mu}
\def\snu{\sim_\nu}
\def\srh{\sim_\rho}

\def\sp{\op{Spec}}

\def\supp{\op{supp}}

\def\t{\theta}

\def\ti{t_\infty}

\newcounter{cs}
\stepcounter{cs}
\newcommand{\casos}{\begin{itemize}}
\newcommand{\fcasos}{\end{itemize}\setcounter{cs}{1}}

\newfont{\tit}{cmr12 scaled \magstep3}

\title[Limit key polynomials]{Invariants of limit key polynomials}

\makeatletter
\@namedef{subjclassname@2010}{%
  \textup{2010} Mathematics Subject Classification}

\author[Alberich]{Maria Alberich-Carrami$\tilde{\mbox{n}}$ana}
\address{Departament de Matem\`atiques,  Universitat Polit\`ecnica de Catalunya, Av. Diagonal 647, Barcelona 0828, Catalonia, Spain}
\email{Maria.Alberich@upc.edu}
\author[F. Boix]{Alberto F. Boix}
\address{Ben Gurion Univ Negev, Dept Math, POB 653, IL-84105 Beer Sheva, Israel}
\email{albertof.boix@gmail.com}
\author[Fern\'andez]{Julio Fern\'andez}
\author[Gu\`ardia]{Jordi Gu\`ardia}
\address{Departament de Matem\`atica Aplicada IV, Escola Polit\`ecnica Superior d'Enginye\-ria de Vilanova i la Geltr\'u, Av. V\'\i ctor Balaguer s/n. E-08800 Vilanova i la Geltr\'u, Catalonia, Spain}
\email{julio.fernandez.g@upc.edu, jordi.guardia-rubies@upc.edu}
\author[Nart]{Enric Nart}
\author[Ro\'e]{Joaquim Ro\'e}
\address{Departament de Matem\`{a}tiques,         Universitat Aut\`{o}noma de Barcelona,         Edifici C, E-08193 Bellaterra, Barcelona, Catalonia, Spain}
\email{nart@mat.uab.cat,\quad jroe@mat.uab.cat}

\date{}

\begin{document}

\begin{abstract}
Let $\nu$ be a valuation of arbitrary rank on the polynomial ring $\kx$ with coefficients in a field $K$.
We prove comparison theorems between MacLane-Vaqui\'e key polynomials for valuations $\mu\le\nu$ and abstract key polynomials for $\nu$.

Also, some results on invariants attached to limit key polynomials are obtained. In particular, if $\chr(K)=0$ we show that all limit key polynomials of unbounded continuous MacLane chains have numerical character equal to one.
\end{abstract}


\maketitle

\section*{Introduction}

Consider a valuation $\nu$ on the polynomial ring $\kx$, with coefficients in a field $K$. Let $\gn$ be its value group.
The graded algebra of $\nu$ is the integral domain
$$
\gr_{\nu}(\kx)=\bigoplus\nolimits_{\alpha\in\g_\nu}\ppa/\ppa^+,
$$
where $\ppa=\{g\in \kx\mid \nu(g)\ge \alpha\}\supset\ppa^+=\{g\in \kx\mid \nu(g)> \alpha\}$.


A \emph{MacLane-Vaqui\'e (MLV) key polynomial} for $\nu$ is a monic polynomial $\phi\in\kx$ whose initial term
\ generates a prime ideal in the graded algebra $\gr_{\nu}(\kx)$, which cannot be generated by the initial term of a polynomial of smaller degree.   

The \emph{degree} of $\nu$ is the minimal degree of a MLV key polynomial for $\nu$.\e


By a celebrated result of MacLane and Vaqui\'e, $\nu$ is the stable limit of a sequence of augmentations of valuations on $\kx$ 
\begin{equation}\label{mu<}
\mu_0<\mu_1<\cdots<\mu_i<\cdots <\nu
\end{equation}
where $\mu_0$ is a valuation of degree one \cite{mcla,Vaq,MLV}.\e

Let $\kpcal_0=\{\phi_0\}$, where $\phi_0\in\kp(\mu_0)$ is any MLV key polynomial of degree one.

These augmentations of valuations can be either \emph{ordinary} or \emph{limit} augmentations. In both cases, certain MLV key polynomials  of the intermediate valuations $\mu_i<\nu$ are involved.

If $\mu_{i-1}<\mu_i$ is an ordinary augmentation, there exists $\phi_i\in\op{KP}(\mu_{i-1})$ such that $\mu_i$ is equal to the truncated valuation $\nu_{\phi_i}$. That is, in terms of $\phi_i$-expansions of polynomials $f\in\kx$, $\mu_i$ acts as follows
$$f=\sum\nolimits_{0\le s}a_s \phi_i^s,\quad \deg(a_s)<\deg(\phi_i)\ \imp\ \mu_i(f)=\min\left\{\nu\left(a_s\phi_i^s\right)\mid 0\le s\right\}.$$
To any such ordinary augmentation step we attach the set 
$$
\kpcal_i=\left\{\phi_i\right\}.
$$

If $\mu_{i-1}<\mu_i$ is a limit augmentation, there exists a countably infinite chain of ordinary augmentations of constant degree
\begin{equation}\label{rho<}
\mu_{i-1}=\rho_0\ < \rho_1\ < \ \cdots \ < \ \rho_i\ < \ \cdots \ <\mu_i,\qquad\qquad \deg(\rho_i)=m,\quad \forall\,i\ge1,
\end{equation}
admitting non-stable polynomials, all of them of degree larger than $m$.

A polynomial $f\in\kx$ is \emph{stable} with respect to the chain if there exists an index $i_0$ such that $\rho_i(f)=\rho_j(f)$ for all $j>i\ge i_0$. In this case, we define this stable value by $\rhi(f)$.

In such a chain $\left(\rho_i\right)_{i\ge0}$ of valuations, the value group is constant except eventually for the valuation $\rho_0$. Let us denote this group by  $$\gi=\g_{\rho_1}=\g_{\rho_2}=\;\cdots\;=\g_{\rho_i}=\;\cdots$$

The set $\kpi(\mu_i)$ of \emph{MLV limit key polynomials} for $\mu_i$ is defined as the set of monic non-stable polynomials of minimal degree. For any $\phi_i\in\kpi(\mu_i)$ we have again $\mu_i=\nu_{\phi_i}$.

To any such limit augmentation step, we attach the totally ordered set 
$$
\kpcal_i=\left\{\chi_1,\dots,\chi_i,\dots\right\}+\left\{\phi_i\right\},
$$
where $\chi_i\in\kp(\rho_{i-1})$ is a MLV key polynomial such that $\rho_i=\nu_{\chi_i}$, and we consider the usual sum ot totally ordered sets.\e

The main result of MacLane-Vaqui\'e states that $\nu$ falls in one, and only one, of the following cases \cite[Thm.4.8]{MLV}.  \e

(1) \ After a finite number $r$ of augmentation steps, we get $\mu_r=\nu$.\e

(2) \ After a finite number $r$ of augmentation steps, $\nu$ is the stable limit  $\nu=\rhi$, of
some countably infinite chain of ordinary augmentations of $\mu_r$  as in (\ref{rho<}) (with $r=i-1$), with the property that all polynomials in $\kx$ are stable.\e

(3) \ It is the stable limit, $\nu=\lim_{i\to\infty}\mu_i$, of a countably infinite chain of mixed augmentations as in (\ref{mu<}), with unbounded degree.\e

We say that $\mu$ has \emph{finite depth} $r$, \emph{quasi-finite} depth $r$, or \emph{infinite depth}, respectively.

The valuations of finite depth are characterized by the condition \cite[Lem.-Def.4.9]{MLV} 
$$\kpn\ne\emptyset\qquad \mbox{or} \qquad\supp(\nu)\ne0.$$ 

If $\nu$ has quasi-finite depth, consider the totally ordered set
$$
\kpcal_{\infty}=\left\{\chi_1,\dots,\chi_i,\dots\right\},
$$
where $\chi_i\in\kp(\rho_{i-1})$ is a MLV key polynomial such that $\rho_i=\nu_{\chi_i}$, for all $i\ge 1$.\e

\as{1.3}
Then, the well-ordered set of polynomials:\vskip-0.1cm

$$
\kpcal=\begin{cases}
\kpcal_0+\cdots+\kpcal_r,&\mbox{ if $\nu$ has finite depth }r,\\
\kpcal_0+\cdots+\kpcal_r+\kpcal_{\infty},&\mbox{ if $\nu$ has quasi-finite depth }r,\\
\kpcal_0+\cdots+\kpcal_i+\cdots,&\mbox{ if $\nu$ has infinite depth},
\end{cases}
$$\vskip0.1cm

\as{1}
\noindent is a \emph{complete system of key polynomials} for $\nu$, as defined by   
F.J. Herrera Govantes, W. Mahboub, M.A. Olalla Acosta and M. Spivakovsky in \cite{hmos}. 
That is, for any $f\in\kx$ there exists $Q\in\kpcal$ such that $\nu(f)=\nu_Q(f)$. As a consequence, for any $\al\in\gn$, the set of polynomials
$$
\kpcal_\al=\left\{a\,Q_1^{n_1}\cdots Q_\ell^{n_\ell}\mid a\in K^*,\ Q_1,\dots,Q_\ell\in\kpcal,\ n_1,\dots,n_\ell\in\N\right\}\cap \pset_\al
$$
is a set of generators of $\pset_\al$ as an additive group.

This property is the motivation for Spivakovsky's strategy to attack the problem of local uniformization \cite{NS2016,SS}. 

Certain \emph{abstract} key polynomials were introduced by J. Decaup, W. Mahboub and M. Spivakovsky
as an intrinsic characterization of the polynomials in $\kpcal$ \cite{Dec}. This idea was developed by Novacoski and Spivakovsky in \cite{NS2018,NS2019}, where they proved some further properties of key polynomials.\e

In this paper, we have a double aim. On one hand, in section \ref{secAKP}, we review some of these results aiming at a determination of which MLV key polynomials of the intermediate valuations $\mu_i$ are abstract key polynomials for $\nu$. We complete in this way some partial results from \cite[Sec.3]{Dec}.     

In section \ref{secLKP} we obtain similar results for limit key polynomials. An abstract limit key polynomial is an element in the well-ordered set $\kpcal$ which does not admit an immediate predecessor. Novacoski and Spivakovsky found an intrinsic characterization  of these polynomials in \cite{NS2018}. We prove that they coincide with the MLV key polynomials of the intermediate valuations $\mu_i$ which are a limit augmentation of the immediate predecessor $\mu_{i-1}$ in (\ref{mu<}).

On the other hand, in section \ref{secInvLKP}, we  obtain some results on invariants attached to limit key polyomials. Our main result in this section is Theorem \ref{ubthm}, where we prove an identity between some of these invariants:
\begin{equation}\label{tibi}
\ti(\phi)\,\bi=\ml(\phi),
\end{equation}
where $\phi$ is a limit key polynomial of any limit augmentation $\mu_{i-1}<\mu_i$ such that the sequence of values $\nu(\rho_i)$ is unbounded in $\gi$.

The invariant $\ml(\phi)$ is the least positive integer $b$ such that $\pb{\phi}\ne0$, where $$\partial_b=\dfrac1{b!}\,\dfrac{\partial^b}{\partial x^b}$$ is the $b$-th formal derivative, which makes sense in any characteristic.

For any $i\ge1$, consider the $\chi_i$-expansion $\phi=\sum_{0\le s}a_{s,i}\chi_i^s$, and let $t_i(\phi)$ be the maximal index $s$ such that $\rho_i(\phi)=\rho_i\left(a_{s,i}\chi_i^s\right)$. This positive integer $t_i(\phi)$ stabilizes for $i$ sufficiently large \cite[Sec.3]{Vaq2004}, \cite[Sec.4]{hmos}.
We denote by $\ti(\phi)$ this stable index, which is known as the \emph{numerical character} of $\phi$.     

Finally, for any $i\ge1$, let $b_i$ be the largest positive integer such that 
$(\nu(\chi_i)-\nu(\prt{b_i}{\chi_i})/b_i$ takes a maximal value in $\gn\otimes\Q$. It is shown in \cite[Sec.7]{hmos} that $b_i$ stabilizes for $i$ sufficiently large, and $\bi$ is this stable value.

As a consequence of (\ref{tibi}), if $\chr(K)=0$, then $\ti(\phi)=\bi=1$, because $\ml(\phi)=1$.

\section{Preliminaries}
\subsection{Valuations on a polynomial ring}\label{subsecValsKx}
Consider a valued field $(K,v)$. Let $k$ be the  residue class field and $\g=v(K^*)$ the value group. Denote  the divisible hull of $\g$ by $$\gq=\g\otimes\Q.$$

Consider an extension $\nu$ of $v$ to the polynomial ring $\kx$ in one indeterminate. That is, for some embedding $\g\hookrightarrow\La$ into another ordered abelian group, we consider a mapping 
$$
\nu\colon \kx\lra \La\infty
$$
whose restriction to $K$ is $v$, and satisfies the following two conditions:\e

(1) \ $\nu(fg)=\nu(f)+\nu(g)$,\qquad$\forall\,f,g\in\kx$.\e

(2) \ $\nu(f+g)\ge\min\{\nu(f),\nu(g)\}$,\qquad$\forall\,f,g\in\kx$.\e

The \emph{support} of $\nu$ is the prime ideal $$\p=\nu^{-1}(\infty)\in\sp(\kx).$$ 

The value group of $\nu$ is the subgroup $\gn\subset \La$ generated by $\nu\left(\kx\setminus\p\right)$.



The valuation $\nu$ induces a valuation on the residue field $\ka(\p)$, field of fractions of $\kx/\p$. Let $k_{\nu}$ be the residue class field of this valuation on $\ka(\p)$.

Clearly, $\ka(0)=K(x)$, while for $\p\ne0$ the field $\ka(\p)$ is a simple finite extension of $K$. 

The extension $\nu/v$ is \emph{commensurable} if $\g_\nu/\g$ is a torsion group. In this case, there is a canonical embedding $\ \g_\nu\hooklongrightarrow \gq$.
All valuations with non-trivial support are commensurable over $v$.




We denote the graded algebra of $\nu$ defined in the Introduction by
$$\ggn=\gr_{\nu}(\kx).$$
If $\nu$ has non-trivial support $\p\ne0$, there is a natural isomorphism of graded algebras
\begin{equation}\label{isomL}
\ggn\simeq \op{gr}_{\bar{\nu}}(\ka(\p)),
\end{equation}
where $\bar{\nu}$ is the valuation on $\ka(\p)$ induced by $\nu$. 

In particular, every non-zero homogeneous element of $\ggn$ is a unit, if $\p\ne0$.

Consider the \emph{initial term} mapping $\inn\colon \kx\to \ggn$, given by $\inn0=0$ and 
$$
\inn g= g+\pset_{\nu(g)}^+\in\pset_{\nu(g)}/\pset_{\nu(g)}^+, 
$$
if $g\ne0$. 
The following definitions translate properties of the action of  $\nu$ on $\kx$ into algebraic relationships in the graded algebra $\ggn$.

\begin{definition}Let $g,\,h\in \kx$.

We say that $g,h$ are \emph{$\nu$-equivalent}, and we write $g\snu h$, if $\ \inn g=\inn h$. 

We say that $g$ is \emph{$\nu$-divisible} by $h$, and we write $h\mmu g$, if $\ \inn h\mid \inn g$ in $\ggn$. 

\end{definition}


\subsection{MacLane-Vaqui\'e key polynomials}\label{subsecMLVKP}

Consider a valuation $\nu$ on $\kx$, extending $v$. 


A polynomial $g\in\kx$ is \emph{$\nu$-irreducible} if $(\inn g)\ggn$ is a non-zero prime ideal. 
 
We say that $g$ is \emph{$\nu$-minimal} if $g\nmid_\nu f$ for all non-zero $f\in \kx$ with $\deg(f)<\deg(g)$.

The property of $\nu$-minimality admits a relevant characterization.

\begin{lemma}\cite[Prop.2.3]{KP}\label{minimal0}
Let $g\in \kx$ be a non-constant polynomial. Let 
$$f=\sum\nolimits_{0\le s}a_sg^s, \qquad  a_s\in \kx,\quad \deg(a_s)<\deg(g) $$ 
be the canonical $g$-expansion of $f\in \kx$.Then, $g$ is $\nu$-minimal if and only if
$$\nu(f)=\min\{\nu(a_sg^s)\mid 0\le s\},\quad \forall\,f\in \kx.$$
\end{lemma}

A \emph{MacLane-Vaqui\'e (MLV) key polynomial} for $\nu$ is a monic polynomial in $\kx$ which is simultaneously  $\nu$-minimal and $\nu$-irreducible. 

The set of MLV key polynomials for $\nu$ will be denoted $\kpn$. 

By the isomorphism of (\ref{isomL}), only valuations with trivial support may have MLV key polynomials.

A MLV key polynomial is necessarily irreducible in $\kx$.

For any $\phi\in\kpn$, we denote by $\cl{\phi}\subset \kpn$ the subset of all MLV key polynomials which are $\nu$-equivalent to $\phi$.

\begin{definition}\label{rele}
If $\kpn\ne\emptyset$, the \emph{degree}  $\deg(\nu)$ is the minimal degree of a MLV key polynomial for $\nu$.
The following subset of $\gn$ is a subgroup: $$\g_{\nu,\deg(\nu)}=\left\{\nu(a)\mid 0\le \deg(a)<\deg(\nu)\right\}.$$  The index
$\ \erel(\nu)=\left(\gn\colon \g_{\nu,\deg(\nu)}\right)$ \
is the \emph{relative ramification index} of $\nu$.
\end{definition}



Consider the subring  of homogeneous elements of degree zero
$$\Delta=\Delta_\nu=\pset_0/\pset_0^+\subset\ggn.$$ 
There are canonical injective ring homomorphisms: 
$$ k\hooklongrightarrow\Delta\hooklongrightarrow k_{\nu}.$$\vskip.1cm

We denote  the algebraic closure of $k$ in $\Delta$ by  $$\kappa=\kappa(\nu)\subset\Delta.$$ This is a subfield such that $\kappa^*=\Delta^*$, the multiplicative group of the units of $\Delta$.

\begin{theorem}\cite[Thm.4.4]{KP}\label{empty}
The set $\kpn$ is empty if and only if all homogeneus elements in $\ggn$ are units. Equivalently, $\nu/v$ is commensurable and $\ka=\dm=\km$ is an algebraic extension of $k$.
\end{theorem}


\begin{theorem}\cite[Thm.4.2]{KP}\label{incomm}
Suppose $\nu/v$ incommensurable. Let $\phi\in\kx$ be a monic polynomial of minimal degree satisfying $\nu(\phi)\not\in\gq$. Then,
$\phi$ is a MLV key polynomial for $\nu$, and $\kpn=\cl{\phi}$.

In this case, $\ka=\dm=\km$ is a finite extension of $k$.
\end{theorem}




\begin{theorem}\cite[Thms.4.5,4.6]{KP}\label{comm}
Suppose that $\nu/v$ is commensurable and $\kpn\ne\emptyset$.  
Let $\phi$ be a MLV key polynomial for $\nu$ of minimal degree $m$. 

Let $e=\erel(\nu)$. Let $a\in\kx$ be a polynomial of degree less than $m$ with $\nu(a)=e\nu(\phi)$, and let $u=\inn a\in\ggn^*$. Then, 
$$
\xi=(\inn \phi)^eu^{-1}\in\Delta
$$
is transcendental over $k$ and satisfies $\Delta=\kappa[\xi]$.

Moreover, the canonical embedding $\Delta\hookrightarrow \km$ induces an isomorphism $\ka(\xi)\simeq \km$. 
\end{theorem}

These comensurable extensions $\nu/v$ admitting MLV key polynomials are called \emph{residually transcendental} valuations on $\kx$.

The pair $\phi,\, u$ determines a  (non-canonical) \emph{residual polynomial operator} 
$$
R=R_{\nu,\phi,u}\colon\;\kx\lra \kappa[y],
$$
whose images are monic polynomials in the indeterminate $y$, which are not divisible by $y$ \cite[Sec.5]{KP}. 
This operator facilitates a complete description of the set $\kpn$.

\begin{theorem}\cite[Prop.6.3]{KP}\label{charKP}
Suppose that $\nu/v$ is commensurable and $\kpn\ne\emptyset$. Let $\phi$ be a MLV key polynomial for $\nu$ of minimal degree $m$. 
A monic $\chi\in\kx$ is a key polynomial for $\nu$ if and  only if either
\begin{enumerate}
\item[(1)] $\deg(\chi)=m$ \,and\; $\chi\snu\phi$, or
\item[(2)] $\deg(\chi)=me\deg(R(\chi))$ \,and\; $R(\chi)$ is irreducible in $\kappa[y]$.
\end{enumerate}
Moreover, $\chi,\,\chi'\in\kpn$ are $\nu$-equivalent if and only if $R(\chi)=R(\chi')$. In this case, $\deg(\chi)=\deg(\chi')$.
\end{theorem}

The set $\kpn/\!\sim_\nu$ is in canonical bijection with the maximal spectrum of $\Delta$ \cite[Thm.6.7]{KP}. 

Since the choice of a pair $\phi,u$ as above determines an isomorphism $\Delta\simeq \ka[y]$, it induces a (non-canonical) bijection between 
$\kpn/\!\sim_\nu$ and the set of monic irreducible polynomials in $\ka[y]$.

\subsection{Chains of valuations}\label{subsecChains}
For any valuation $\mu$ on $\kx$ taking values in a subgroup of $\gn$, we say that $$\mu\le\nu\quad\mbox{ if }\quad\mu(f)\le\nu(f),\qquad \forall\,f\in\kx.$$ 

Suppose that $\mu<\nu$.  Let $\phmn$ be the set of all monic polynomials $\phi\in\kx$ of minimal degree among those satisfying $\mu(\phi)<\nu(\phi)$.

By a well known result of MacLane-Vaqui\'e \cite[Lem.1.15]{Vaq}, any $\phi\in\phmn$ is a MLV key polynomial for $\mu$ and satisfies
$$
\mu(f)=\nu(f)\sii \phi\nmid_\mu f,\qquad \forall\,f\in\kx.
$$

Actually, $\phmn$ is an element in $\kpm/\!\smu$. That is, $\phmn=[\phi]_\mu$, for all $\phi\in\phmn$ \cite[Cor.2.6]{MLV}. 
We define $$\deg\left(\phmn\right)=\deg(\phi)\quad \mbox{ for any }\quad \phi\in\phmn.$$

If we have a chain $\mu<\eta<\nu$ of valuations, we have $\phmn=\Phi_{\mu,\eta}$ \cite[Cor.2.6]{MLV}. In particular,
\begin{equation}\label{transitive}
\mu(f)=\nu(f)\sii \mu(f)=\eta(f),\qquad  \forall\,f\in\kx.
\end{equation}

\section{Abstract key polynomials}\label{secAKP}

Consider a valuation $\nu$ on $\kx$. 

Abstract key polynomials for $\nu$ were introduced by J. Decaup, W. Mahboub and M. Spivakovsky in \cite{Dec}
as an intrinsic characterization of the members of a \emph{complete system of key polynomials} defined by F.J. Herrera Govantes, W. Mahboub, M.A. Olalla Acosta and M. Spivakovsky in \cite{hmos}. 

In this section, we review some of these results. Our aim is to find exactly what MLV key polynomials of the intermediate valuations $\mu_i$ are abstract key polynomials for $\nu$, completing in this way some partial results from \cite[Sec.3]{Dec}.     


\subsection{Invariants of polynomials with respect to the valuation $\nu$}

We denote by $\N$ the set of positive integers.
For any $b\in \N$, consider the linear differential operator $\partial_b$ on $\kx$, defined by Taylor's formula:
$$
f(x+y)=\sum_{0\le b}\pb fy^b, \quad\ \forall\,f\in\kx,
$$
where $y$ is another indeterminate. Note that
$$
\pb{x^n}=\comb{n}{b}x^{n-b},\quad\ \forall\, n\in\N,
$$
if we agree that $\comb{n}{b}=0$ whenever $n<b$.

Let $f\in\kx$ be a polynomial of positive degree.  Denote 
$$
\ml(f)=\mbox{least $b\in\N$ such that }\prt{b}f\ne0. 
$$

Clearly, $\ml(f)=1$ if $\chr(K)=0$. 
If $\chr(K)=p$, then $\ml(f)=p^r$ is the largest power of $p$ such that $f$ belongs to $K[x^{p^r}]$.

This integer $\ml(f)$ is an intrinsic datum of $f$. We are interested in some data that may be attached to $f$ in terms of the valuation $\nu$.

\begin{definition}
Let $f\in\kx\setminus K$ such that $\nu(f)<\infty$. We define
$$
\ep(f)=\max\left\{\dfrac{\nu(f)-\nu(\pb{f})}b\ \Big|\ b\in\N\right\}\in\left(\gn\right)_\Q.
$$
If $\nu(f)=\infty$, we define $\ep(f)=\infty$. 
\end{definition}

In particular, if $\nu(f)<\infty$ we have
\begin{equation}\label{epsdef}
 \nu(\pb{f})\ge\nu(f)-b\,\ep(f),\quad\ \forall\, b\ge0,
\end{equation}
and we define $I(f)\subset \N$ to be the set of positive integers for which equality holds.

If $\nu(f)=\infty$ and $f$ is irreducible, we agree that $I(f)=\{\ml(f)\}$. 
Otherwise, the set  $I(f)$ is not defined.\e

\noindent{\bf Examples. }\e

\begin{itemize}
\item If $\deg(f)=1$, then $\ep(f)=\nu(f)$ and $I(f)=\{1\}$. \e
\item If $a\in K^*$, then $\ep(af)=\ep(f)$ and $I(af)=I(f)$.\e
\item If $f$ is monic and $b=\deg(f)$, then $\pb{f}=1$. Hence, $\nu(f)/\deg(f)\le\ep(f)$.\e
\item If $b\not\in[\ml(f),\deg(f)]$, then $\pb{f}=0$. Hence, $I(f)\subset[\ml(f),\deg(f)]$. \e
\item For $a\in K$ and $f=(x-a)^n$, we have $\ep(f)=\nu(x-a)$ and
$$
\begin{array}{ll}
I(f)=[1,n]\cap\N,&\quad\mbox{if }\ \chr(k)=0,\\ I(f)=\left\{b\in [1,n]\cap\N\ \Big|\ p\nmid \comb{n}{b}\right\},&\quad\mbox{if }\ \chr(k)=p.
\end{array}
$$
\end{itemize}


Novacoski and Spivakovski found an interesting interpretation of $\ep(f)$ in \cite{NS2019}.

\begin{proposition}\label{r=maxI}
Let $f\in \kx$ be a monic polynomial such that $\nu(f)<\infty$. Let $\op{Z}(f)\subset\overline{K}$ be the multiset of roots of $f$ in an algebraic closure of $K$. 

For any extension $\bar{\nu}$ of $\nu$ to $\overline{K}[x]$, we have
\begin{equation}\label{ep=max}
\ep(f)=\max\{\bar{\nu}(x-\t)\mid \t\in \op{Z}(f)\}.
\end{equation}

Moreover, the multiplicity of $\ep(f)$ in the multiset $\{\bar{\nu}(x-\t)\mid \t\in \op{Z}(f)\}$ is equal to $\max(I(f))$.
\end{proposition}

\begin{proof}
The equality (\ref{ep=max}) is proved in \cite[Prop.3.1]{NS2019}. We reproduce the proof because we need it to prove the second statement.

Let $\op{Z}(f)=\{\t_1,\dots,\t_n\}$, $\delta=\max\{\bar{\nu}(x-\t_i)\mid 1\le i\le n\}$. Let $r$ be the multiplicity of $\delta$ in this multiset.
For any integer $1\le s\le n$, we have
$$
\prt{s}{f}=\sum_{J}\left(\prod_{i\not\in J}(x-\t_i)\right)\ \imp\ \bar{\nu}\left(\prt{s}{f}\right)\ge\min_J\left\{\sum_{i\not\in J}\bar{\nu}(x-\t_i)\right\},
$$
where $J$ runs on all subsets of $[1,n]\cap\N$ of cardinality $s$.

For $s=r$, the set $J_0=\{i\in [1,n]\cap\N\mid \bar{\nu}(x-\t_i)=\delta\}$ is the unique subset of cardinality $r$ for which the term $\sum_{i\not\in J_0}\bar{\nu}(x-\t_i)$ takes the minimal value. Hence,
$$
\nu(\prt{r}{f})=\sum_{i\not\in J_0}\bar{\nu}(x-\t_i).
$$
This implies 
$$
\nu(f)-\nu(\prt{r}{f})=\sum_{i\in J_0}\bar{\nu}(x-\t_i)=r\delta.
$$

For any $s\ne r$, let $J$ be one of the subsets of cardinality $s$ for which $\sum_{i\not\in J}\bar{\nu}(x-\t_i)$ takes the minimal value. Then,
\begin{equation}\label{ineq}
\nu(f)-\nu(\prt{r}{f})\le\sum_{i\in J}\bar{\nu}(x-\t_i)\le s\delta.
\end{equation}
This proves that $\ep(f)=\delta$ and $r$ belongs to $I(f)$.

Now, if $s>r$, there is at least one index $i\in J$ for which  $\bar{\nu}(x-\t_i)<\delta$. Hence, we get an strict inequality in (\ref{ineq}). This proves that $s\not\in I(f)$.
\end{proof}

\begin{corollary}\label{maxepfg}
For any two $f,g\in\kx\setminus K$, we have 
\begin{equation}\label{epfg}
\ep(fg)=\max\{\ep(f),\,\ep(g)\}.
\end{equation}
Moreover, if $\ep(f)<\ep(g)<\infty$, then $I(fg)=I(g)$.
\end{corollary}

\begin{proof}
The equality  (\ref{epfg}) follows immediately from Proposition \ref{r=maxI}.

Suppose  $\ep(f)<\ep(g)$, and denote $\ep=\ep(g)=\ep(fg)$. For any $b\in \N$,
$$
\pb{fg}=\sum_{j=0}^b\prt{j}{f}\prt{b-j}{g}\imp\nu\left(\pb{fg}\right)\ge\min\{\nu(\prt{j}{f})+\nu(\prt{b-j}{g})\mid 0\le j\le b\}.
$$

For any index $j>0$ the inequality (\ref{epsdef}) shows that
$$
\nu(\prt{j}{f})+\nu(\prt{b-j}{g})\ge \nu(f)-j\ep(f)+\nu(g)+(b-j)\ep>\nu(fg)-b\,\ep.
$$
For the index $j=0$,
$$
\nu(f)+\nu(\prt{b}{g})\ge \nu(f)+\nu(g)-b\,\ep=\nu(fg)-b\,\ep,
$$
and equality holds if and only if $b\in I(g)$. This proves that $I(fg)=I(g)$.
\end{proof}\e

\begin{remark}\rm
If $\supp(\nu)=f\kx$, then Proposition \ref{r=maxI} still holds for $f$. In fact, there must be a root $\t\in\op{Z}(f)$ such that $\bar{\nu}(x-\t)=\infty$. Then, necessarily  $\supp(\bar{\nu})=(x-\t)\overline{K}[x]$. Hence, the multiplicity of $\infty$ in the multiset $\{\bar{\nu}(x-\t)\mid \t\in \op{Z}(f)\}$ is equal to the multiplicity of $\t$ in the multiset $\op{Z}(f)$, which coincides with  $\ml(f)$ because $f$ is irreducible.
\end{remark}

\subsection{Abstract key polynomials. Basic properties}
Following the criterion of \cite{NS2018}, we drop the adjective ``abstract" and talk simply of key polynomials for the valuation $\nu$.

\begin{definition}
A monic $Q\in\kx$ is a \emph{key polynomial for $\nu$} if for all $f\in\kx$,
 it satisfies
$$
0<\deg(f)<\deg(Q)\ \imp\ \ep(f)<\ep(Q).
$$
\end{definition}

\noindent{\bf Examples}

\begin{itemize}
\item All monic polynomials of degree one are key polynomials for $\nu$.
\item If $\op{supp}(\nu)=\phi\kx$ for a monic $\phi\in\kx$, then $\phi$ is a key polynomial for $\nu$.

On the other hand, we saw in section \ref{subsecMLVKP} that $\kpn=\emptyset$ if $\op{supp}(\nu)\ne0$.
\end{itemize}\e

By Corollary \ref{maxepfg}, all key polynomials are irreducible in $\kx$.

Let $p$ be the \emph{characteristic exponent} of the valued field $(K,v)$. That is, 
$$
p=\begin{cases}
\chr(k),&\mbox{ if }\chr(k)>0,\\
1,&\mbox{ if }\chr(k)=0.
\end{cases}
$$

\begin{proposition}\label{bIp}\cite[Prop.2.4]{NS2018}
If $Q\in\kx$ is a key polynomial, then all elements in $I(Q)$ are a power of the characteristic exponent $p$. 
\end{proposition}

The next basic property of key polynomials is a generalization of \cite[Prop.10]{Dec}.

\begin{lemma}\label{abcd}
Let $Q\in\kx$ be a key polynomial for $\nu$, and let $f\in\kx$ be non-constant polynomial such that $\ep(f)<\ep(Q)$.
Consider the division with remainder in $\kx$:
$$
f=a+qQ, \qquad \deg(a)<\deg(Q).
$$
Then, \ $\nu(f)=\nu(a)<\nu(qQ)$.
\end{lemma}

\begin{proof}
Suppose that $\nu(qQ)\le \nu(a)$. Then, we have  $\nu(qQ)\le \nu(f)$ as well. Let us show that this leads to a contradiction.

Since $Q$ is a key polynomial, $\ep(a)<\ep(Q)$. For any $b\in I(qQ)$, we have
$$
\nu\left(\pb{f}\right)\ge \nu\left(f\right)-b\,\ep(f)>\nu\left(f\right)-b\,\ep(Q)\ge \nu\left(qQ\right)-b\,\ep(Q),
$$
$$
\nu\left(\pb{a}\right)\ge \nu\left(a\right)-b\,\ep(a)>\nu\left(a\right)-b\,\ep(Q)\ge \nu\left(qQ\right)-b\,\ep(Q).
$$
Since $\pb{qQ}=\pb{f}-\pb{a}$, we deduce 
$$
\nu(qQ)-b\,\ep(qQ)=\nu\left(\pb{qQ}\right)>\nu\left(qQ\right)-b\,\ep(Q).
$$
This implies $\ep(qQ)<\ep(Q)$, contradicting Corollary \ref{maxepfg}.
\end{proof}

\begin{corollary}\label{units}
Let $Q\in\kx$ be a key polynomial for $\nu$, and let $f\in\kx$ be non-constant polynomial such that $\ep(f)<\ep(Q)$.
Then, $\inn f$ is a unit in $\ggn$.
\end{corollary}

\begin{proof}
Corollary \ref{maxepfg} shows that $f$ is not divisible by $Q$ in $\kx$.
Since $Q$ is irreducible, there is a B\'ezout identity:
$$
aQ+bf=1,\quad \deg(b)<\deg(Q).
$$
By Lemma \ref{abcd}, $bf\sim_\nu 1$, or equivalently, $(\inn b)(\inn f)=\inn 1$ in $\ggn$.
\end{proof}\e

Let $Q\in\kx$ be a  monic polynomial. Consider the function
$$
\nu_Q\colon \kx \lra \gn\infty,\qquad \nu_Q(f)=\min\{\nu(a_sQ^s)\mid 0\le s\},
$$
where $f=\sum_{0\le s}a_sQ^s$ is the canonical $Q$-expansion of $f$ (cf. Lemma \ref{minimal0}).

We denote by $S_{\nu,Q}(f)$ the set of indices $s$ for which $\nu(a_sQ^s)=\nu_Q(f)$.

This function $\nu_Q$ is not necessarily a valuation (see Lemma \ref{zero} below). However, it is a valuation if $Q$ is a key polynomial.

\begin{proposition}\label{vQval}\cite[Prop.12]{Dec}
If $Q$ is a key polynomial for $\nu$, then $\nu_Q$ is a valuation on $\kx$ such that $\nu_Q\le \nu$.  
\end{proposition}

\subsection{Comparison between abstract and MLV key polynomials}

\begin{proposition}\label{A->MLV}
If $Q$ is a key polynomial for $\nu$, then either $\nu(Q)=\infty$, or $Q$ is a MLV key polynomial  of minimal degree for $\nu_Q$.  
\end{proposition}

\begin{proof}
Suppose $\nu(Q)<\infty$.

By the very definition of $\nu_Q$, the key polynomial $Q$ satisfies the criterion of $\nu_Q$-minimality of Lemma \ref{minimal0}.
In particular,  $Q\nmid_\nu 1$, so that $\inn Q$ is not a unit in $\ggn$.

On the other hand, for all polynomials $f\in\kx$ of degree less than $\deg(Q)$, the element $\inn f$ is a unit in $\ggn$. In fact, this follows from $\ep(f)<\ep(Q)$, by Corollary \ref{units}.

Hence, $Q$ is a MLV key polynomial of minimal degree for $\nu_Q$ \cite[Thm.3.2+Prop.3.5]{KP}.
\end{proof}\e

The rest of the section is devoted to analyze what MLV polynomials of valuations $\mu\le\nu$ are (abstract) key polynomials for $\nu$. The next two results are crucial for this purpose. 

\begin{proposition}\label{nuQnu}\cite[Lem.14+Prop.15]{Dec}, \cite[Prop.2.7]{NS2018}
Let $Q\in\kx$ be a key po\-ly\-nomial for $\nu$. For all $f\in\kx$, the following hold:

(i) \quad\, For all $b\in\N$ we have
\begin{equation}\label{epdefQ}
\nu_Q\left(\pb{f}\right)\ge \nu_Q(f)-b\,\ep(Q). 
\end{equation}

(ii) \ \; If $S_{\nu,Q}(f)\ne\{0\}$, then equality holds in (\ref{epdefQ}) for some $b\in\N$.\e

(iii) \ If equality holds in (\ref{epdefQ}) for $b\in\N$ and $\nu_Q\left(\pb{f}\right)=\nu\left(\pb{f}\right)$, then $\ep(f)\ge \ep(Q)$.

If in addition, $\nu(f)>\nu_Q(f)$, then $\ep(f)>\ep(Q)$.
\end{proposition}


\begin{proposition}\label{eporder} \cite[Prop.20+Lem.24]{Dec}
Let $Q, Q'\in\kx$ be key polynomials for $\nu$. Then, $$\ep(Q)\le\ep(Q')\sii \nu_Q\le\nu_{Q'}.$$
In this case, $\nu_{Q'}(Q)=\nu(Q)$.
Moreover, $\ep(Q)<\ep(Q')$ if and only if $\nu_Q(Q')<\nu(Q')$.
\end{proposition}




From now on, we fix a valuation $\mu$ on $\kx$ with values in the group $\gn$ and satisfying $$\mu\le\nu.$$ 
Let us first determine for what MLV key polynomials for $\mu$ the truncation $\nu_\phi$ is a valuation.

\begin{lemma}\label{zero}
Suppose $\mu<\nu$, and take $\phi\in\kpm$.
\begin{enumerate}
\item $\phi\in\phmn\ \imp\ \nu_\phi\mbox{ is a valuation and }\ \mu<\nu_\phi\le \nu$. 
\item $\phi\not \in\phmn,\ \deg(\phi)\le \deg\left(\phmn\right)\ \imp\ \nu_\phi=\mu$. 
\item $\deg(\phi)> \deg\left(\phmn\right)\ \imp\ \nu_\phi$ \ is not a valuation. 
\end{enumerate}
\end{lemma}

\begin{proof}
(1)  For any $\phi\in\phmn$, the function $\nu_\phi$ coincides with the augmented valuation $[\mu;\phi,\nu(\phi)]$ \cite[Sec.1]{Vaq}. The inequalities $\mu<\nu_\phi\le \nu$ are obvious.\e

(2) \ Suppose $\phi\not \in\phmn$ and $\deg(\phi)\le \deg\left(\phmn\right)$. Then, 
$$\mu(\phi)=\nu(\phi),\qquad \mu(a)=\nu(a),\ \forall\,a\in\kx \mbox{ with }\deg(a)<\deg(\phi). $$
Since $\phi$ is $\mu$-minimal, Lemma \ref{minimal0} shows that
$$
\mu(f)=\min\{\mu\left(a_s\phi^s\right)\mid 0\le s\}=\min\{\nu\left(a_s\phi^s\right)\mid 0\le s\}=\nu_\phi(f),
$$
for all $f\in\kx$ with $\phi$-expansion $f=\sum_{0\le s} a_s\phi^s$.\e

(3) \ Let $\phmn=[Q]_\mu$, and suppose $\deg(\phi)> \deg(Q)$ and $\nu_\phi$ is a valuation. Let us derive a contradiction.

By the definition of $\nu_\phi$, our polynomial $\phi$ satisfies the criterion of $\nu_\phi$-minimality of Lemma \ref{minimal0}. Hence, $\op{in}_{\nu_\phi} \phi$ is not a unit in $\gg_{\nu_\phi}$. By Theorem \ref{empty}, the set $\op{KP}(\nu_\phi)$ is not empty.

Therefore, we may apply \cite[Thm.3.9]{KP} to both valuations $\mu$ and $\nu_\phi$. For all monic polynomials $f\in\kx$ we have
$$
\dfrac{\mu(f)}{\deg(f)}\le\dfrac{\mu(\phi)}{\deg(\phi)},\qquad \dfrac{\nu_\phi(f)}{\deg(f)}\le\dfrac{\nu_\phi(\phi)}{\deg(\phi)},
$$
and equality holds if and only if $f$ is $\mu$-minimal, or $\nu_\phi$-minimal, respectively. 

If we apply these inequalities to $f=Q$ we get a contradiction: 
$$
\dfrac{\mu(Q)}{\deg(Q)}<\dfrac{\nu(Q)}{\deg(Q)}=\dfrac{\nu_\phi(Q)}{\deg(Q)}\le \dfrac{\nu_\phi(\phi)}{\deg(\phi)}=\dfrac{\nu(\phi)}{\deg(\phi)}=\dfrac{\mu(\phi)}{\deg(\phi)}=\dfrac{\mu(Q)}{\deg(Q)},
$$
where the last equality holds because $Q$ is $\mu$-minimal.
\end{proof}\e

\begin{lemma}\label{u}\cite[Lem.2.11]{NS2018}
Let $Q$ be a key polynomial for  $\nu$ such that $\nu_Q<\nu$. Then, all polynomials in $\Phi_{\nu_Q,\nu}$ are key polynomials for $\nu$.
\end{lemma}

\begin{proposition}\label{dos}
Let $\mu$ be a valuation on $\kx$ such that $\mu\le \nu$. Then all MLV key polynomials for $\mu$ of minimal degree are key polynomials for $\nu$. 
\end{proposition}

\begin{proof}
If $\mu=\nu$ and $\kpn=\emptyset$, the statement of the proposition is empty. Therefore, in the case $\mu=\nu$ we may assume that $\kpn\ne\emptyset$, 

We proceed by induction on $\deg(\mu)$. If $\deg(\mu)=1$, the statement is obvious because all monic polynomials of degree one are key polynomials.

Suppose $\deg(\mu)\ge2$ and the statement holds for all valuations $\rho<\nu$ of degree less than $\deg(\mu)$. Let $\phi\in\kpm$ be a MLV key polynomial for $\mu$ of minimal degree $\deg(\phi)=\deg(\mu)$. 

Since $\kpm\ne\emptyset$, the theorem of MacLane-Vaqui\'e 
shows that $\mu$ is the augmentation of a valuation $\rho$ of smaller degree. Let us discuss in an independent way the cases in which $\mu$ is an ordinary or a limit augmentation of $\rho$.\e

\noindent{\bf Ordinary augmentation. }We have $\mu=[\rho; \chi,\mu(\chi)]=\nu_\chi$, for a certain MLV key polynomial $\chi\in\kpr$ satisfying $\mu(\chi)>\rho(\chi)$, which becomes a MLV key polynomial of minimal degree for $\mu$ \cite[Cor.7.3]{KP}. 

In particular, $\deg(\phi)=\dg(\chi)$ and $\mu(\phi)=\mu(\chi)$ \cite[Thm.3.9]{KP}. Let us write $\phi=\chi+a$, with $a\in\kx$ of degree less than $\deg(\chi)$.
Since $\Phi_{\rho,\mu}=[\chi]_\rho$, we have  
$$
\rho(a)=\mu(a)\ge\mu(\chi)>\rho(\chi).
$$
Hence, $\phi\srh\chi$, so that $\phi\in\Phi_{\rho,\mu}$. 

Now, let $Q\in\kpr$ be a MLV key polynomial for $\rho$ of minimal degree; that is, $\deg(Q)=\deg(\rho)<\deg\mu)=\deg(\phi)$.
By the induction hypothesis, $Q$ is a key polynomial for $\nu$. Since $Q\not\in \Phi_{\rho,\mu}$, Lemma \ref{zero} shows that $\nu_Q=\rho$. Thus, Lemma \ref{u} shows that $\phi\in\Phi_{\rho,\mu}$ is a key polynomial for $\nu$.\e

\noindent{\bf Limit augmentation. }The valuation $\rho$ is the initial valuation of a continuous MacLane chain of valuations of constant degree $m$:
$$
\rho=\rho_0\ \stackrel{\chi_1,\be_1}\lra\ \rho_1\ \lra\ \cdots\ \stackrel{\chi_i,\be_i}\lra\ \rho_i\ \lra\ \cdots
$$
Each $\rho_i=[\rho_{i-1};\chi_i,\be_i]=\nu_{\chi_i}$ is an ordinary augmentation. All polynomials $\chi_i$ have degree $m$ and belong to $\kp(\rho_{i-1})\cap\kp(\rho_i)$.  

All polynomials $f\in\kx$ of degree less than or equal to $m$ are stable; that is, $\rho_{i}(f)=\mu(f)=\nu(f)$ for all $i$ sufficiently large.


Also, the chain admits polynomials which are not stable, and  we have $\mu=\nu_\varphi$ for some monic non-stable $\varphi\in\kx$ of minimal degree, which becomes a MLV key polynomial of minimal degree for $\mu$ \cite[Cor.7.13]{KP}. See section \ref{subsecLMLV} for a more precise definition of a limit augmentation.

In particular, $\deg(\varphi)=\deg(\phi)$ and $\mu(\varphi)=\mu(\phi)$ \cite[Thm.3.9]{KP}. Let us write $\phi=\varphi+a$, with $a\in\kx$ of degree less than $\deg(\varphi)$. By the minimality of $\deg(\varphi)$, the polynomial $a$ is stable; that is, for some index $i_0$ we have
$$
\rho_i(a)=\mu(a)\ge \mu(\varphi)>\rho_i(\varphi),\qquad \forall\,i\ge i_0.
$$
Hence, $\phi\sim_{\rho_i}\varphi$ for all $i\ge i_0$. This implies that $\phi$ is non-stable too:
$$
\rho_i(\phi)=\rho_i(\varphi)<\mu(\varphi)=\mu(\phi),\qquad \forall\,i\ge i_0.
$$

By the induction hypothesis, all $\chi_i$ are key polynomials for $\nu$. 
\ Take any $b\in[1,\deg(\phi)]\cap\N$. Since $\deg\left(\pb{\phi}\right)<\deg(\phi)$, the polynomial $\pb{\phi}$ is stable. Take $i$ sufficiently large so that 
$$
\rho_i\left(\pb{\phi}\right)=\mu\left(\pb{\phi}\right)=\nu\left(\pb{\phi}\right),\qquad \forall\, b\in\N.
$$

By \cite[Sec.4]{hmos}, or \cite[Sec.3]{Vaq2004}, the integers $\max(S_{\nu,\chi_i}(\phi))$ are all positive, and stabilize for a sufficiently large index $i$. In particular, $S_{\nu,\chi_i}(\phi)\ne\{0\}$ for all $i$. 
By (iii) of Proposition \ref{nuQnu}, $\ep(\phi)>\ep(\chi_i)$ for all $i$ sufficiently large.

Now, take any $f\in\kx$ with $\deg(f)<\deg(\phi)$. Since $f$ and $\pb{f}$ are stable, we may take $i$ sufficiently large so that
$$
\rho_i(f)=\nu(f),\qquad \rho_i\left(\pb{f}\right)=\nu\left(\pb{f}\right),\qquad \forall\,b\in\N.
$$
By (i) of Proposition \ref{nuQnu}, for all $b\in\N$ we have
$$
\dfrac{\nu(f)-\nu\left(\pb{f}\right)}b=
\dfrac{\rho_i(f)-\rho_i\left(\pb{f}\right)}b\le \ep(\chi_i)<\ep(\phi).
$$
Thus, $\phi$ is a key polynomial for $\nu$.
\end{proof}\e

Lemma \ref{zero} exhibited the first examples of MLV key polynomials for $\mu$ that are not (abstract) key polynomials for $\nu$. The next lemma offers some more examples.

\begin{lemma}\label{tres}
Suppose $\mu<\nu$, and take $\phi\in\kpm$.
If $\phi\not \in\phmn$ and $\deg(\phi)>\deg(\mu)$, then $\phi$ is not a key polynomial for $\nu$. 
\end{lemma}

\begin{proof}
If $\deg(\phi)>\deg(\phmn)$, the lemma follows from Lemma \ref{zero} and Proposition \ref{vQval}. Suppose $\deg(\mu)<\deg(\phi)\le\deg(\phmn)$.

Let $\phi_0$ be a MLV key polynomial of minimal degree $\deg(\phi_0)=\deg(\mu)$. 
By Lemma \ref{zero}, $\nu_\phi=\mu=\nu_{\phi_0}$.

By Proposition \ref{dos}, $\phi_0$ is a key polynomial for $\mu$. Hence, $\phi$ cannot be a key polynomial because it would have $\ep(\phi)>\ep(\phi_0)$, contradicting Proposition \ref{eporder}. 
\end{proof}\e

We may summarize the results obtained so far in the next two theorems.

\begin{theorem}\label{main}
Suppose that $\mu<\nu$ and $\phi\in\kpm$. Then, $\phi$ is a key polynomial for $\nu$ if and only if it satisfies one of the following two conditions.
\begin{enumerate}
\item $\phi\in\phmn$,
\item $\phi\not \in\phmn$ \ and \ $\deg(\phi)=\deg(\mu)$.
\end{enumerate}

In the first case, $\nu_\phi=[\mu;\phi,\nu(\phi)]$. In the second case, $\nu_\phi=\mu$.
\end{theorem}

\begin{theorem}\label{mu=nu}
Let $\phi\in\kpn$. Then, $\phi$ is a key polynomial for $\nu$ if and only if $\deg(\phi)=\deg(\nu)$. In this case, $\nu_\phi=\nu$.
\end{theorem}

By Theorem \ref{charKP}, two $\mu$-equivalent MLV key polynomials for $\mu$ have the same degree. Hence, the next result follows immediately from Theorems \ref{main} and \ref{mu=nu}.

\begin{corollary}\label{classes}
 Suppose that $\mu\le\nu$ and $\phi\in\kpm$. If $\phi$ is a key polynomial for $\nu$, then all polynomials in $[\phi]_\mu$ are key polynomials for $\nu$ too. 
\end{corollary}

\begin{corollary}\label{Qmaxep}
Let $\phi\in\kpn$ of minimal degree. Then, $\ep(\phi)\ge\ep(f)$ for all $f\in\kx$.  
\end{corollary}

\begin{proof}
By Theorem \ref{mu=nu}, $\phi$ is a key polynomial for $\nu$ and $\nu_\phi=\nu$.


The result follows from (i) of Proposition \ref{nuQnu}.
\end{proof}\e

Also, these results lead to another characterization of abstract key polynomials.

\begin{theorem}\label{More}
 Let $\nu$ be a valuation on $\kx$, and $Q\in\kx$ a monic polynomial.
 The following conditions are equivalent.
 
 \begin{enumerate}
 \item  $Q$ is a key polynomial for $\nu$.
 \item $\nu_Q$ is a valuation and either $\supp(\nu)=Q\kx$, or $Q$ is a MLV key polynomial for $\nu_Q$ of minimal degree.
 \item  $\nu_Q$ is a valuation and $Q$ has minimal degree among all monic polynomials $f\in\kx$ satisfying $\nu_f=\nu_Q$.
 \end{enumerate}
\end{theorem}

\begin{proof}
(1) $\Rightarrow$ (2) follows from Propositions \ref{vQval} and \ref{A->MLV}.

(2) $\Rightarrow$ (1) follows from Proposition \ref{dos}.

(2) $\Rightarrow$ (3). Let $f\in\kx$ be a monic polynomial such that $\nu_f=\nu_Q$. If $\supp(\nu)=Q\kx$, then $\nu(f)=\nu_f(f)=\nu_Q(f)=\infty$, so that $f$ is a multiple of $Q$. 

Suppose that $Q$ is a MLV key polynomial for $\nu_Q$ of minimal degree. By Lemma \ref{minimal0}, $f$ is $\nu_Q$-minimal; thus, $\deg(f)$ is a multiple of $\deg(Q)$ \cite[Prop.3.7]{KP}.

(3) $\Rightarrow$ (2). Suppose $\supp(\nu)\ne Q\kx$. Then, (3) implies that $\nu(Q)<\infty$. By Lemma \ref{minimal0}, $Q$ has minimal degree among all $\nu_Q$-minimal polynomials. 

Let $Q_0$ be a MLV key polynomial for $\nu_Q$ of minimal degree. By \cite[Prop.3.7]{KP}, $Q=Q_0+a$ for some $a\in\kx$ with $\deg(a)<\deg(Q_0)$ and $\nu_Q(a)\ge \nu_Q(Q_0)$. Hence, either $Q\sim_{\nu_Q}Q_0$ (if $\nu_Q(a)> \nu_Q(Q_0)$), or $\deg(R(Q))=1$ (if $\nu_Q(a)=\nu_Q(Q_0)$). By Theorem \ref{charKP}, $Q$ is a MLV key polynomial for $\nu_Q$ of minimal degree.
\end{proof}\e

A key polynomial $Q$ for $\nu$ is said to be  \emph{maximal} if $\nu_Q=\nu$. These key polynomials admit the following characterization.

\begin{corollary}\label{More2}
 Let $\nu$ be a valuation on $\kx$, and $Q\in\kx$ a monic polynomial.
 The following conditions are equivalent.
 
 \begin{enumerate}
 \item  $Q$ is a maximal key polynomial for $\nu$.
 \item Either $\supp(\nu)=Q\kx$, or $Q$ is a MLV key polynomial for $\nu$ of minimal degree.
 \item  $\ep(Q)\ge \ep(f)$ for all polynomials $f\in\kx$, and $Q$ has minimal degree with this property.
 \end{enumerate}
\end{corollary}

\begin{proof}
Theorem \ref{More} shows that (1) and (2) are equivalent. \e

(2) $\Rightarrow$ (3). Corollary \ref{Qmaxep} shows that $\ep(Q)\ge \ep(f)$ for all polynomials $f\in\kx$.

Since $Q$ is a key polynomial for $\nu$, for any polynomial $f$ of smaller degree $\ep(f)$ cannot be maximal because $\ep(f)<\ep(Q)$. \e

(3) $\Rightarrow$ (1). By definition, a monic polynomial of minimal degree for which $\ep(Q)$ takes a maximal value is a key polynomial for $\nu$.
Finally, $\nu_Q=\nu$ by Proposition \ref{eporder}.
\end{proof}\e

As shown in the Introduction, only the valuations $\nu$ of finite depth admit maximal key polynomials. 

\section{Comparison of MacLane-Vaqui\'e and abstract limit key polynomials}\label{secLKP}

\subsection{MacLane-Vaqui\'e limit key polynomials}\label{subsecLMLV}
Let us recall the definition of MacLane-Vaqui\'e (MLV) limit key polynomials \cite[Sec.2]{Vaq}, \cite[Sec.3]{MLV}.

Let $\rho$ be a valuation on $\kx$ admitting MLV key polynomials.

\begin{definition}
A \emph{continuous MacLane chain} of $\rho$, of stable degree $m$, is a countably infinite chain of ordinary augmentations:  
\begin{equation}\label{chainrho}
\rho=\rho_0\ \stackrel{\chi_1,\be_1}\lra\ \rho_1\ \lra\ \cdots\ \stackrel{\chi_i,\be_i}\lra\ \rho_i\ \lra\ \cdots\qquad\quad \rho_i=[\rho_{i-1};\chi_i,\be_i],\quad \forall\,i\ge1,
\end{equation}
such that  $\deg(\chi_i)=m$ and $\chi_{i+1}\nmid_{\rho_i}\chi_i$,  for all $i\ge1$.
\end{definition}

By \cite[Prop.7.2]{KP} each $\chi_i\in\op{KP}(\rho_{i-1})$ becomes a MLV key polynomial of minimal degree of $\rho_i$. In particular, $\deg(\rho_i)=m$ for all $i\ge1$.

A polynomial $f\in\kx$ is \emph{stable} with respect to the chain,  if for some index $i_0$ we have
$$
\rho_{i}(f)=\rho_{i_0}(f), \quad\ \forall\,i\ge i_0.
$$
This stable value is denoted $\rhi(f)$.

By the equivalence (\ref{transitive}), a non-stable polynomial $f$ satisfies necessarily 
$$
\rho_i(f)<\rho_j(f),\qquad \forall\,i<j.
$$
Let $\mi$ be the minimal degree of a non-stable polynomial. We agree that $\mi=\infty$ if all polynomials are stable.

The following properties hold for all continuous MacLane chains \cite[Lem.3.3]{MLV}
\begin{itemize}
 \item $\mi\ge m$.
 \item All polynomials $\chi_i$ are stable.
 \item For all $i\ge1$, $\rho_i$ is residually transcendental and $\g_{\rho_i}=\g_{\rho_1}$. 
\end{itemize}

The common value grup $\gi:=\g_{\rho_i}$ for all $i\ge1$ is called the \emph{stable value group} of the continuous MacLane chain. Note that $\be_i\in\gi$ for all $i\ge1$.\e

\noindent{\bf Remark. }In \cite[Sec.3]{MLV} it was supposed that $\rho$ is residually transcendental too, and $\g_{\rho}=\gi$. We omit these conditions on $\rho$ because they are irrelevant for the analysis of the limit behaviour of the chain.\e

Any continuous MacLane chain $\left(\rho_i\right)_{i\ge0}$ falls in one of the following three cases:

\begin{enumerate}
\item[(a)] It has a \emph{stable limit}.
That is, $\mi=\infty$ and the function $\rhi$ is a valuation on $\kx$. This valuation is commensurable and satisfies $\kp(\rhi)=\emptyset$.\e

\item[(b)] It is \emph{inessential}.
That is, $\mi=m$.\e

\item[(c)] It is \emph{essential}.
That is, $\mi>m$.\e
\end{enumerate}

Let $\nu$ be a valuation on $\kx$ such that $\rho_i<\nu$ for all $i\ge0$.

If $\left(\rho_i\right)_{i\ge0}$ is inessential and $f\in\kx$ is a non-stable polynomial of degree $m$, then the ordinary augmentation $\mu=[\rho; f,\nu(f)]$ satisfies 
$$
\rho_i<\mu\le\nu,\qquad \forall\,i\ge0.
$$
In other words, $\mu$ is closer to $\nu$ than any $\rho_i$, and we may access to $\mu$ from $\rho$ by a single augmentation. In the terminology of \cite{MLV},  we may avoid any reference to the continuous MacLane chain $\left(\rho_i\right)_{i\ge0}$ along the process of constructing a MacLane-Vaqui\'e chain of valuations for $\nu$.

In the terminology of \cite{hmos}, all key polynomials $\chi_i$ may be replaced by the single key polynomial $f$ in any complete system of key polynomials for $\mu$.

This justifies why we call it ``inessential".\e

Only the essential continuous MacLane chains admit (non-fake) limit key polynomials. From now on, we suppose that our chain $(\rho_i)_{i\ge0}$ is essential.\e

We define the set of MLV limit key polynomials for $(\rho_i)_{i\ge0}$:
$$
\kpi=\kpi\left((\rho_i)_{i\ge0}\right),
$$
as the set of monic non-stable polynomials in $\kx$ of minimal degree $\mi$.


Take $\phi\in\kpi$.
Let $\gi\hookrightarrow \La$ be an embedding of ordered groups, and choose $\ga\in\La\infty$ such that 
$$\ga>\rho_i(\phi),\quad\ \forall\,i\ge0.$$  
We may consider a limit augmentation
$$
\mu_{\phi,\ga}=[(\rho_i)_{i\ge0};\phi,\ga],
$$
which on $\phi$-expansions $f=\sum_{0\le s}a_s\phi^s$ acts as follows:
$$
\mu_{\phi,\ga}(f)=\min\{\rhi(a_s)+s\ga\mid 0\le s\}=\min\{\mu_{\phi,\ga}\left(a_s\phi^s\right)\mid 0\le s\}.
$$

This function $\mu_{\phi,\ga}$ is a valuation on $\kx$ which satisfies $\mu_{\phi,\ga}>\rho_i$ for all $i\ge0$. \e

Let $\nu$ be a valuation on $\kx$ such that $\nu>\rho_i$ for all $i\ge0$. 

For every stable polyomial $f$ one has $\nu(f)=\rhi(f)$. Thus, $\nu(\chi_i)=\be_i$ for all $i\ge 1$.

A monic polynomial $f\in\kx$ of degree $m$ such that $\nu(f)>\be_i$ for all $i\ge 1$ would be non-stable. 
In fact, if $f$ were stable, there would exist an index $i$ such that $\rho_i(f)=\nu(f)$. This is impossible because $\rho_i(f)\le\rho_i(\chi_i)=\be_i$ by \cite[Thm.3.9]{KP}.

Since we are assuming that our continuous MacLane chain is essential, we have  
\begin{equation}\label{betaicofinal}
(\be_i)_{i\ge0}\quad\mbox{ is cofinal in the set }\quad\left\{\nu(f)\mid f\in\kx\mbox{ monic, }\deg(f)=m\right\}.
\end{equation}

It is easy to see that any MLV limit key polynomial $\phi\in\kpi$ is a key polynomial for $\nu$.

In fact, take $\ga=\nu(\phi)$. For all pair of indices $i<j$, we have $\rho_i(\phi)<\rho_j(\phi)\le\nu(\phi)=\ga$. The limit augmented valuation $\mu_{\phi,\ga}$ clearly satisfies $\mu_{\phi,\ga}\le\nu$. 
By \cite[Cor.7.13]{KP}, $\phi$ is a MLV key polynomial for $\mu_{\phi,\ga}$ of minimal degree. Thus, our claim follows from Proposition \ref{dos}.

\subsection{Abstract limit key polynomials}
Let $\nu$ be a valuation on $\kx$.
Novacoski and Spivakovsky define in \cite{NS2018} an (abstract) limit key  polynomial for $\nu$ as a monic polynomial $Q\in\kx$ for which there exists a key polynomial $Q_-$ satisfying the following conditions.
\begin{enumerate}
\item[(K1)] $\deg(Q_-)=\deg\left(\Phi_{\nu_{Q_-},\nu}\right)$. 
\item[(K2)] the set $\{\nu(\chi)\mid\chi\in \Phi_{\nu_{Q_-},\nu}\}$ has no maximal element.
\item[(K3)] $\nu_{\chi}(Q)<\nu(Q)$ for all $\chi\in\Phi_{\nu_{Q_-},\nu}$.
\item[(K4)] $Q$ has minimal degree among all polynomials satisfying (K3).
\end{enumerate}

\begin{proposition}\label{LMLV->LAKP}
Let $(\rho_i)_{i\ge0}$ be an essential continuous MacLane chain as in (\ref{chainrho}).
Let $\nu$ be a valuation on $\kx$ such that $\nu>\rho_i$ for all $i\ge0$. Then, all MLV limit key polynomials for $(\rho_i)_{i\ge0}$ are limit key polynomials for $\nu$.   
\end{proposition}

\begin{proof}
Let  $\phi\in\kpi$. We may take $Q_-=\chi_1$, which is a key polynomial by Proposition \ref{dos}. In section \ref{subsecChains} we saw that
$$
[\chi_2]_{\rho_1}=\Phi_{\rho_1,\rho_2}=\Phi_{\rho_1,\rho_i}=\Phi_{\rho_1,\nu},\qquad \forall\,i\ge 1.
$$
Since $[\chi_1]_{\rho_1}\ne [\chi_2]_{\rho_1}$, Lemma \ref{zero} shows that $\nu_{\chi_1}=\rho_1$. Since, $\deg(\chi_1)=\deg(\chi_2)=\deg\left(\Phi_{\rho_1,\nu}\right)$, condition (K1) is satisfied.

Condition (K2) follows from (\ref{betaicofinal}). 

Take $\chi\in\Phi_{\nu_{\chi_1},\nu}=\Phi_{\rho_1,\nu}=[\chi_2]_{\rho_1}$. Let $\mu=[\rho_1;\chi,\nu(\chi)]$. Since $\chi$ is a MLV key polynomial of minimal degree for $\mu$, and $\chi\not\in\phmn$, we have $\nu_\chi=\mu$ by Lemma \ref{zero}. Since $\deg(\chi)=m$, the property (\ref{betaicofinal}) shows that there exists an index $i$ such that $\nu(\chi)<\be_i$. This implies $\mu<\rho_i$, and from this we deduce $\mu(\phi)\le\rho_i(\phi)<\nu(\phi)$. This proves (K3).

Finally, any monic polynomial $Q$ satisfying (K3) is non-stable, Thus, $\deg(Q)\ge \mi=\deg(\phi)$. This proves (K4).
\end{proof}\e

The converse statement holds too.

\begin{proposition}\label{LAKP->LMLV}
Let $Q\in\kx$ be a limit key polynomial for $\nu$. Then, $Q$ is a MLV limit key polynomial for some essential continuous MacLane chain $(\rho_i)_{i\ge0}$.   
\end{proposition}

\begin{proof}
Let $Q_-\in\kx$ be a key polynomial satisfying conditions (K1)--(K4). Define $\rho=\rho_0=\nu_{Q_-}$ and $m=\deg(Q_-)$. 
All elements in $\Phi_{\rho,\nu}$ are MLV key polynomials for $\rho$. By (K1) these polynomials have degree $m$. 

By (K2) there exists a sequence $(\chi_i)_{i\ge1}$ of polynomials in $\Phi_{\rho,\nu}$ such that the sequence $\be_i=\nu(\chi_i)$ is strictly increasing  and cofinal in the set $\nu\left(\Phi_{\rho,\nu}\right)$. 

By Lemma \ref{u}, $\chi_i$ is a key polynomial for $\nu$, for all $i$. Let $\rho_i=\nu_{\chi_i}$ for all $i$. By construction, $\chi_j$ belongs to $\Phi_{\rho_i,\nu}$ for all $j>i$; thus, $\chi_j$ is a MLV key polynomial for $\rho_i$.

By the very definition of truncation and ordinary augmentation, $\rho_i=[\rho_{i-1};\chi_i,\be_i]$. Also, $\chi_i\not\in\Phi_{\rho_i,\nu}$ implies that $[\chi_i]_{\rho_i}\ne[\chi_{i+1}]_{\rho_i}$. 

Therefore, $(\rho_i)_{i\ge0}$ is a continuous MacLane chain of $\rho$, which is not inessential by (K2). 

By (K3) and (K4), $Q$ is non-stable of minimal degree. Hence,  $(\rho_i)_{i\ge0}$ is essential and $Q$ is a MLV key polynomial.

\end{proof}

\section{Invariants of limit key polynomials}\label{secInvLKP}

\subsection{Basic invariants of continuous MacLane chains}\label{subsecBasicInvs}
Consider a fixed essential continuous MacLane chain $\left(\rho_i\right)_{i\ge 0}$ as in (\ref{chainrho}). 

Our aim in this section is to study certain invariants of MLV limit  key polynomials, introduced in \cite[Sec.3]{Vaq2004} and \cite[Sec.4]{hmos}.

Let $\phi\in\kpi$, and let $n=\lfloor \mi/m\rfloor$. Denote the canonical $\chi_i$-expansion of $\phi$ by
$$
\phi=a_{n,i}\,\chi_i^n+a_{n-1,i}\,\chi_i^{n-1}+\cdots+a_{1,i}\,\chi_i+a_{0,i},\qquad \forall\,i\ge 1.
$$

The index $t_i(\phi)=\max\left(S_{\rho_i,\chi_i}(\phi)\right)$ is always positive and decreases as $i$ grows. Thus, it stabilizes for $i$ sufficiently large. The stable value is known as the \emph{numerical character} of $\phi$. We denote it by\footnote{  This invariant is denoted $t$ in \cite{Vaq2004} and $\delta$ in \cite{hmos}.} 
$$
\ti=\ti(\phi).
$$
This integer is a power of the characteristic exponent $p$ of the valued field $(K,v)$ \cite[Sec.7]{hmos}.

Let $i_0$ be an index which stabilizes $\ti$. Let us denote $t=\ti$ for simplicity. It is easy to check that the image of the coefficient $a_{t,i}$ in the graded algebra stabilizes too:
$$
a_{t,i}\sim_{\rho_k}a_{t,j},\qquad\forall\,i_0\le i<j\le k.
$$
In particular, it determines a stable value
$$
\ali=\ali(\phi)=\rho_i(a_{t,i})=\rhi(a_{t,i})\in\gi,\qquad \forall\, i\ge i_0.
$$
Since $\rho_i(\phi)=\rho_i\left(a_{t,i}\,\chi_i^{t}\right)$, we have
$$
\rho_i(\phi)=\ali+\ti\,\be_i,\qquad\forall\,i\ge i_0.
$$

\begin{proposition}\label{powerphi}
Take $\phi\in\kpi$ and let $i_0$ be an index that stabilizes $t=\ti$. Then, 
$$
\phi\sim_{\rho_i}a_{t,j}\,\chi_{j}^t,\qquad \forall\,i_0< i<j.
$$
\end{proposition}

\begin{proof}
For all pair of indices $k<\ell$ we have
$$
\rho_k(\chi_\ell)=\be_k<\be_\ell=\rho_\ell(\chi_\ell)\ \imp\ \chi_\ell\in \Phi_{\rho_k,\rho_\ell}.
$$
Thus, $\chi_\ell$ is a MLV key polynomial for $\rho_k$, for all $k\le\ell$. In particular, it is $\rho_k$-minimal and Lemma \ref{minimal0} shows that
\begin{equation}\label{jminimal}
\rho_k(\phi)=\min\left\{\rho_k\left(a_{s,\ell}\,\chi_\ell^s\right)\mid 0\le s\right\},\qquad \forall\,0\le k\le\ell.
\end{equation}

Now, denote $\al=\ali$ and take any pair of indices $j>i>i_0$. Let us apply (\ref{jminimal}) for $k=i$, $\ell=j$. For $s=t$ we get the minimal value  
$$
\rho_i\left(a_{t,j}\,\chi_j^t\right)=\al+t\be_i=\rho_i(\phi).
$$
The proposition will be proved if we show that $\rho_i\left(a_{s,j}\,\chi_j^s\right)$ takes an strictly larger value for all indices $s\ne t$.   

For $s<t$, we  apply (\ref{jminimal}) for $k=j=\ell$. We get
\begin{align*}
\rhi(a_{s,j})+s\be_j=&\;\rho_j\left(a_{s,j}\,\chi_j^s\right)\ge \rho_j(\phi)=\al+t\be_j\\
\imp&\;\rhi(a_{s,j})\ge \al+(t-s)\be_j>\al+(t-s)\be_i.
\end{align*}

For $s>t$, we  apply (\ref{jminimal}) for $k=i_0$, $\ell=j$. Since $t-s$ is a negative integer,
we get
$$
\rhi(a_{s,j})+s\be_{i_0}=\rho_{i_0}\left(a_{s,j}\,\chi_j^s\right)\ge \rho_{i_0}(\phi)=\al+t\be_{i_0}
$$\vskip-.6cm
\begin{equation}\label{serveix}
\qquad\qquad\qquad\imp\rhi(a_{s,j})\ge \al+(t-s)\be_{i_0}>\al+(t-s)\be_i,
\end{equation}

In both cases, we deduce that
$$\rho_i\left(a_{s,j}\,\chi_j^s\right)=  \rhi(a_{s,j})+s\be_i>\al+t\be_i=\rho_i(\phi).$$
\end{proof}

\subsection*{Residual polynomial operators of a continuous MacLane chain}
From now on, we freely use the definition and properties of the residual polynomial operators introduced in  \cite[Sec.5]{KP}.
Write $$\chi_{i+1}=\chi_i+a_i,\qquad u_i=\op{in}_{\rho_i}a_i,\qquad \forall\,i\ge1.$$
These polynomials $a_i\in\kx$ have degree $\deg(a_i)<m$; hence, they have a stable value $\rhi(a_i)=\rho(a_i)$ for all $i$.

Since $\be_i=\rho_{i+1}(\chi_i)<\rho_{i+1}(\chi_{i+1})=\be_{i+1}$, we necessarily have 
$$\rhi(a_i)=\rho_{i+1}(a_i)=\be_i,\qquad\forall\,i\ge1.$$

For $i\ge1$, all valuations $\rho_i$ have relative ramification index equal to one (cf. Definition \ref{rele}). Thus, we may consider residual polynomial operators
$$
R_i=R_{\rho_i,\chi_i,u_i}\colon \kx\lra \ka_i[y],\qquad i\ge 1,
$$
where $\ka_i$ is the maximal subfield of $\Delta_{\rho_i}$ (cf. Section \ref{subsecMLVKP}).

For this normalization of the residual polynomial operator we have
$$
R_i(a)=1,\qquad R_i(\chi_i)=1,\qquad R_i(\chi_{i+1})=y+1,\qquad \forall\,i\ge 1,
$$
for all $a\in\kx$ with $\deg(a)<m$.

Since the residual operator is multiplicative \cite[Cor.5.4]{KP}, we deduce immediately from Proposition \ref{powerphi} and \cite[Cor.5.5]{KP} that
$$R_i(\phi)=R_i(a_{t,i+1})R_i(\chi_{i+1})^t=(y+1)^t.$$ 
This result may be deduced from \cite[Prop.4.2]{hmos} too.

\begin{corollary}\label{Sphi}
Take $\phi\in\kpi$ and let $i_0$ be an index that stabilizes $\ti$. Then, 
$$
S_{\rho_i,\chi_i}(\phi)=\{0,\ti\},\qquad \forall\,i\ge i_0.
$$
\end{corollary}

\begin{proof}
 If $\chr(k)=0$, then $\ti=1$ and the statement is obvious.
 
 If $\chr(k)=p>0$, then $\ti=p^e$ for some $e\ge0$, so that $R_i(\phi)=y^{\ti}+1$. Now, by the very definition of $R_i$, the coefficient of degree $j$ of  $R_i(\phi)$ is zero if and only if $j\not\in S_{\rho_i,\chi_i}(\phi)$. Hence, the statement follows.  
\end{proof}

\subsection*{Intrinsic invariants of continuous MacLane chain}

We are ready to show that the invariants $\ti$, $\ali$ are independent of the choice of the MLV limit key polynomial $\phi$.

\begin{lemma}\label{kp=}
For any two $\phi,\,\varphi\in\kpi$ there exists an index $i_0$ such that
$$
\phi\sim_{\rho_i}\varphi,\qquad\forall\,i\ge i_0.
$$ 
\end{lemma}

\begin{proof}
Write $\phi=\varphi+a$ with $a\in\kx$ of degree less than $\mi$. Since $a$ is stable, there exists an index $i_0$ such that $\rho_i(a)=\rhi(a)$ for all $i\ge i_0$.
We want to show that
$$
\rho_i(a)>\rho_i(\phi)\quad\ \forall\,i\ge i_0.
$$
In fact, $\rho_i(a)\le \rho_i(\phi)$ leads to a contradiction:
$$
\rho_j(a)=\rho_i(a)\le \rho_i(\phi)<\rho_j(\phi),\qquad \forall\,j>i,
$$
which implies that $\varphi$ would be stable: $\rho_j(\varphi)=\rho_j(a)=\rhi(a)$ for all $j>i$.
\end{proof}\e

The next result follows immediately from Proposition \ref{powerphi} and Lemma \ref{kp=}.

\begin{corollary}\label{invinv}
For all $\phi,\varphi\in\kpi$ we have $\ti(\phi)=\ti(\varphi)$ and $\ali(\phi)=\ali(\varphi)$. 
\end{corollary}


Let us recall another intrinsic invariant $\bi$ of the chain.

Take any valuation $\nu$ on $\kx$ such that $\nu>\rho_i$ for all $i$. For instance, any limit augmentation of $\left(\rho_i\right)_{i\ge0}$. 

In \cite[Sec.7]{hmos} it is shown that for a sufficiently large index $j_0$ one has:
$$
I(\chi_j)=\{\bi\},\qquad \forall\, j\ge j_0,
$$
for a certain positive integer $\bi$.

Since all polynomials $\chi_j$ and all their derivatives $\pb{\chi_j}$ are stable, it is clear that $\bi$ does not depend on the choice of the valuation $\nu$.

By Proposition \ref{bIp}, $\bi$ is a power of the characteristic exponent $p$ of $(K,v)$.\e

On the other hand, all $\chi_j$ are key polynomials for $\nu$ such that $\nu_{\chi_j}=\rho_j$ by Theorem \ref{main}.
Let us denote $\ep_j=\ep(\chi_j)\in\left(\gi\right)_\Q$. 
In \cite[Cor.7.3]{hmos} it is proved that 
$$
 \rhi(\prt{\bi}{a_j})> \rhi(\prt{\bi}{\chi_j})=\rhi(\prt{\bi}{\chi_{j+1}}),\qquad \forall\,j\ge j_0.
$$

In particular, we may consider another invariant of the essential continuous MacLane chain $\left(\rho_i\right)_{i\ge0}$, independent of $j$ and the choice of $\nu$:$$\di:=\rhi(\prt{\bi}{\chi_j})\in\gi.$$ 

As a consequence we get a direct formula for the variation of $\ep_j$:
\begin{equation}\label{di}
\be_j-\bi\ep_j=\di=\be_{j+1}-\bi\ep_{j+1}\ \imp\ \ep_{j+1}-\ep_j=\dfrac1{\bi}\,\left(\be_{j+1}-\be_j\right).
\end{equation}

Finally, let us quote a basic relationship between these invariants.

\begin{lemma}\label{bt>=ml}
For any $\phi\in\kpi$, we have $\ti\bi\ge \ml(\phi)$. 
\end{lemma}

\begin{proof} 
Let $j$ be a sufficiently large index so that it stabilizes both $\ti$ and $\bi$. Recall that $\chi_j$ is a key polynomial for $\nu$ such that $\nu_{\chi_j}=\rho_j$. 
By Corollary \ref{Sphi}, $S_{\rho_j,\chi_j}(\phi)=\{0,\ti\}$. Let $b=\ti\bi$. By \cite[Prop.6.1]{hmos} or \cite[Prop. 14]{Dec}, 
$$
\rho_j(\pb{\phi})=\rho_j(\phi)-b\ep_j.
$$
In particular, $\pb{\phi}\ne0$, so that $b\ge\ml(\phi)$.
\end{proof}

\subsection{Vertically bounded continuous MacLane chains}\label{subsecVB}
Let us recall Hahn's embedding theorem for ordered groups. A basic reference for this result is \cite{Rib}. 

Let $\La$ be an abelian (totally) ordered group. A subgroup $H\subset \La$ is \emph{convex} if it satisfies
$$
0<\be<\ga, \quad \ga\in H\ \imp\ \be\in H,
$$
for all $\be,\ga\in \La_{>0}$.

For any $\ga\in\La$ we denote by $H_\ga$ the convex subgroup generated by $\ga$. That is, $H_\ga$ is the intersection of all convex subgroups of $\La$ that contain $\ga$. The convex subgroups of the form $H_\ga$  are said to be \emph{principal}.

The principal convex subgoups of $\La$ are totally ordered by inclusion. 
Let us denote by
$$
I=\pcv(\La)
$$
the set of non-zero convex principal subgroups of $\La$, ordered by decreasing inclusion.

Formally, we consider $I$ as an abstract totally ordered set parameterizing the principal convex subgoups. For any $i\in I$ we denote by $H_i$ the corresponding principal convex subgroup. Note that
$$
i\le j\ \sii\ H_i\supset H_j.
$$

Denote by $\rlex$ the Hahn product; that is, $\rlex\subset\R^I$ is the subgroup of the cartesian product $\R^I$ formed by the elements $\ga=\left(x_i\right)_{i\in I}$  whose support
$$
\supp(\ga)=\{i\in I\mid x_i\ne0\}\subset I
$$
is a well-ordered subset, with respect to the ordering induced by $I$. It makes sense to consider the lexicographical ordering in $\rlex$.

By Hahn's theorem, there is an embedding of ordered groups 
$$
\La\hooklongrightarrow\La_\Q\hooklongrightarrow\rlex,
$$
such that the embedding $\La_\Q\hooklongrightarrow\rlex$ is immediate; that is, it determines an identification of  the skeleton of both ordered grups. 
In particular, the natural mapping
$$
\pcv(\La)\lra \pcv\left(\rlex\right), \quad H_\ga\ \mapsto \left(H_{\ga}\right)_\R=\mbox{ convex subgroup of $\rlex$ generated by $\ga$} 
$$
is an isomorphism of ordered sets.

\begin{definition}\label{VBdef}
Consider a strictly increasing sequence of positive elements in $\La$,
$$
S=\{\ga_1<\ga_2<\cdots<\ga_n<\cdots\}\subset \La_{>0}.
$$

Let $H_S$ be the convex subgroup of $\La$ generated by $S$. 

We say that $S$ is \emph{vertically bounded} (VB) if $S$ admits an upper bound in $H_S$.

We say that $S$ is \emph{horizontally bounded} (HB) if $S$ has no upper bounds in $H_S$, but it admits an upper bound in $\La$.

We say that $S$ is \emph{unbounded} (UB) if $S$ admits no upper bounds in $\La$.
\end{definition}

Clearly, any such sequence $S$ falls in one, and only one, of the three cases VB, HB or UB.

Horizontally bounded sequences occur only in ordered groups of rank greater than one.\e

The next table displays some examples in the ordered group $\La=\Q^2_{\op{lex}}$. 

In this case, all convex subroups are principal and $I=\{1,2\}$. The non-zero convex subgroups are
\ $H_1=\La$, \ $H_2=\{0\}\times\Q$.\bs

\begin{center}
\as{1.3}
\begin{tabular}{|c|c|c|}
 \hline
$S$&$H_S$&boundness\\\hline 
$\left(0,1-(1/n)\right)$&$\{0\}\times\Q$&VB\\\hline
$\left(1,n\right)$&$\La$&VB\\\hline
$\left(0,n\right)$&$\{0\}\times\Q$&HB\\\hline
$\left(n,0\right)$&$\La$&UB\\
 \hline
\end{tabular}
\end{center}

\e

\begin{lemma}\label{afitacio}
Let $S=\left(\ga_n\right)_{n\in\N}$ be a strictly increasing sequence of positive elements in $\La$.
Let $H_S$ be the convex subgroup of $\La$ generated by $S$. Then, the following conditions are equivalent.
\begin{enumerate}
\item $S$ is vertically bounded. 
\item For all $q\in\Q$, $q>1$, there exists $n\in\N$ such that $q\ga_n>S$ in $\La_\Q$.
\end{enumerate}
In this case, $H_S$ is a principal convex subgroup.
\end{lemma}

\begin{proof}
It is obvious that (2) implies (1). From $q\ga_n>S$ we deduce that $N\ga_n>S$ for any integer $N\ge q$. 

Let us show that (1) implies (2). If $\ga\in H_S$ satisfies $\ga>S$, then $S\subset H_\ga$ by the convexity of $H_\ga$. Hence, $H_S=H_\ga$ is a principal convex subgroup.

Let $i\in I$ such that $H_i=H_S$. Then, 
$$
(H_S)_\R=\left\{(x_j)\in\rlex\mid x_j=0,\ \forall\, j<i\right\}\subset\rlex.
$$
Thus, we may write
$$
\ga=(0\cdots0\;x\,\star\star\cdots),\qquad \ga_n=(0\cdots0\;x_n\,\star\star\cdots),\quad\forall\,n\in\N,
$$
where $x,x_n\in\R$ are the $i$-th coordinates. They satisfy $0\le x_n\le x$ for all $n\in\N$, and $x_n>0$ for some $n$ (otherwise $S$ would not generate $H_i$). 

Consider $b=\sup(x_n\mid n\in\N)$. 
For any given $q\in\Q$, $q>1$, there exists $n\in\N$ such that $b<qx_n$. Hence, $q\ga_n>S$.
\end{proof}



\begin{definition}\label{rhovbdef}
Let $\left(\rho_i\right)_{i\ge0}$ be an essential continuous MacLane chain as in (\ref{chainrho}). Let $i_0$ be the first index that stabilizes $\ti$, and consider the strictly increasing sequence $$S=\left(\be_i-\be_{i_0}\right)_{i>i_0}.$$ 

We say that the chain $\left(\rho_i\right)_{i\ge0}$ is vertically bounded, horizontally bounded or unbounded according to the boundness status of $S$ introduced in Definition \ref{VBdef}.
\end{definition}

\begin{theorem}\label{vbthm}
Let $\left(\rho_i\right)_{i\ge0}$ be an essential continuous MacLane chain as in (\ref{chainrho}). If $(\rho_i)_{i\ge0}$ is vertically bounded, then
$\mi=m\ti$ and $\ali=0$.
\end{theorem}

\begin{proof}
Denote $t=\ti$, $\al=\ali$, and let $i_0$ be the first index that stabilizes $\ti$.

The sequence $S=\left(\be_i-\be_{i_0}\right)_{i>i_0}$ admits an upper bound in the principal convex subgroup generated by $S$. By Lemma \ref{afitacio}, there exists an index $j>i_0$ such that
\begin{equation}\label{fitat}
\dfrac{t+1}t\,\left(\be_j-\be_{i_0}\right)>\left(\be_i-\be_{i_0}\right),\qquad \forall\,i>i_0. 
\end{equation}

Take any $\phi\in\kpi$, and let $n=\lfloor \mi/m\rfloor$. Consider the canonical $\chi_j$-expansion of $\phi$,
$$
\phi=a_{n,j}\,\chi_i^n+a_{n-1,j}\,\chi_i^{n-1}+\cdots+a_{1,j}\,\chi_i+a_{0,j}.
$$
 
\noindent{\bf Claim. } $\phi\sim_{\rho_i} a_{t,j}\,\chi_j^t+\cdots+a_{1,j}\,\chi_j+a_{0,j},\quad \forall\,i>i_0$.\e

To prove the Claim we must show that
$$\rho_i\left(a_{s,j}\,\chi_j^s\right)>\rho_i(\phi)=\al+t\be_i,\qquad \forall\,s>t,\quad\forall\,i>i_0.$$   
This holds whenever $i<j$ by Proposition \ref{powerphi}. Thus, we may assume that $i\ge j$. In this case, $\rho_i(\chi_j)=\be_j$.

Take any $s>t$. We saw in (\ref{serveix}) that $\rhi(a_{s,j})\ge\al+(t-s)\be_{i_0}$.
Hence, 
$$
\rho_i\left(a_{s,j}\,\chi_j^s\right)=\rhi(a_{s,j})+s\be_j\ge\al+(t-s)\be_{i_0}+s\be_j=\al+t\be_{i_0}+s\left(\be_j-\be_{i_0}\right).
$$
We want to show that $\al+t\be_{i_0}+s\left(\be_j-\be_{i_0}\right)>\al+t\be_i$, which amounts to
$$
s(\be_j-\be_{i_0})>t(\be_i-\be_{i_0}),
$$
and this follows from (\ref{fitat}). This ends the proof of the Claim.\e

By the Claim, the polynomial  $F=a_{t,j}\,\chi_j^t+\cdots+a_{1,j}\,\chi_j+a_{0,j}$ is non-stable. By the minimality of $\mi=\deg(\phi)$, we must have $F=\phi$.

Since the coefficients $a_{s,j}$ have degree less than $m=\deg(\rho_1)$, those which are non-zero determine units in the graded algebra $\gg_{\rho_1}$. Conversely, any unit in $\gg_{\rho_1}$ is the initial term of a polynomial of degree less than $m$ \cite[Prop.3.5]{KP}. Therefore, there exist polynomials $b,c_0,\dots,c_{t-1}\in\kx$, all of degree less than $m$, such that:
$$
ba_{t,j}\sim_{\rho_1}1,\qquad ba_{s,j}\sim_{\rho_1}c_s,\qquad \forall\,0\le s<t.
$$
Since $\rho_1(c_s)=\rhi(c_s)$ for all $s$, we have 
$$
ba_{t,j}\sim_{\rho_i}1,\qquad ba_{s,j}\sim_{\rho_i}c_s,\qquad \forall\,0\le s<t,
$$
for all $i\ge1$. By the Claim, we deduce that
$$
b\phi\sim_{\rho_i}\chi_j^t+c_{t-1}\chi_j^{t-1}+\cdots +c_0,\qquad \forall\, i>i_0.
$$
Since $b\phi$ is clearly non-stable, this implies that the polynomial of degree $mt$, $$\chi_j^t+c_{t-1}\chi_j^{t-1}+\cdots +c_0,$$ is non-stable too. By the minimality of $\mi=\deg(\phi)=\deg(a_{t,j})+mt$, we deduce that $\deg(a_{t,j})=0$, which implies $a_{t,j}=1$ because $\phi$ is monic.

This proves that $\mi=mt$ and $\al=\rhi(a_{t,j})=0$.
\end{proof}

\begin{corollary}
For all VB essential continuous MacLane chains, we have $\ti>1$. 

Therefore, there are no VB essential continuous MacLane chains at all, if $\chr(k)=0$.
\end{corollary}

Theorem \ref{vbthm} was proved for $\rho$ of finite rank in \cite[Sec.3]{Vaq2004}, and for $\rho$ of rank one in \cite[Sec.5]{hmos}. Actually, both proofs are valid for arbitrary rank, once the right definition of vertically bounded chain is introduced. We followed the approach of Vaqui\'e in \cite{Vaq2004}.   

\subsection{Invariants of unbounded continuous MacLane chains}\label{subsecUnbounded}

\begin{theorem}\label{ubthm}
Let $\left(\rho_i\right)_{i\ge0}$ be an essential continuous MacLane chain as in (\ref{chainrho}). If the sequence $(\be_i)_{i\ge0}$ is unbounded in $\gi$, then
$$\ti\bi=\ml(\phi),\qquad \forall\,\phi\in\kpi.$$
\end{theorem}

\begin{proof}
Let $\nu$ be any valuation on $\kx$ such that $\nu>\rho_i$ for all $i$. Recall that all $\chi_i$ are key polynomials for $\nu$ such that $\nu_{\chi_i}=\rho_i$.  
Denote $\ep_i=\ep(\chi_i)$ for all $i$. 

Let $\phi\in\kpi$. Denote 
$$
b=\ml(\phi),\quad t=\ti, \quad\al=\ali.
$$

Since $\pb{\phi}$ has degree less than $\mi$, it is a stable polymomial.
Let $i_0$ be any index that stabilizes $\ti$, $\bi$ and $\rhi(\pb{\phi})$. By definition,
$$
\ep_i=\dfrac{\nu(\chi_i)-\nu(\prt{\bi}{\chi_i})}{\bi}=\dfrac{\be_i-\di}{\bi},\qquad \forall\, i\ge i_0,
$$
where $\di$ is the invariant introduced in (\ref{di}).

By (i) of Proposition \ref{nuQnu}, for all $i\ge i_0$ we have
$$
\rho_i(\pb{\phi})\ge \rho_i(\phi)-b\ep_i=\al+t\be_i-b\ep_i=\al+t\be_i-\dfrac b{\bi}\left(\be_i-\di\right)=\al+\left(t-\dfrac b{\bi}\right)\be_i+\dfrac{b\di}{\bi}.
$$
From this inequality we deduce
$$
\left(t-\dfrac b{\bi}\right)\be_i\le \rhi(\pb{\phi})-\al-\dfrac{b\di}{\bi},\qquad  \forall\, i\ge i_0.
$$

Since $\pb{\phi}\ne0$, we have necessarily $t\bi\le b$. Otherwise, the sequence $(\be_i)_{i\ge0}$ would admit an upper bound in $\left(\gi\right)_\Q$, and hence in $\gi$, against our assumption.

This proves  $t\bi\le b$, and the equality follows from Lemma \ref{bt>=ml}.
\end{proof}

\begin{corollary}\label{chK=0}
If $\chr(K)=0$ and the sequence $\left(\be_i\right)_{i\ge0}$ is unbounded in $\gi$,  then $$\ti=\bi=1.$$ 
\end{corollary}

\begin{proof}
If $\chr(K)=0$, then $\ml(\phi)=1$.
\end{proof}\e

In W. Mahboub PhD thesis \cite{M}, some examples of continuous MacLane chains and limit key polynomials are exhibited. Among the HB ones, there are some examples in which the inequality $\ti\bi\ge\ml(\phi)$ is an equality (Examples 5.3.1, 5.3.2, and 5.3.3), and one where it is an strict inequality (Example 5.3.4).

On the other hand, any monic irreducible polynomial $\phi\in \kx$ which determines an extension of $K$ with defect, is a limit key polynomial of a suitable continuous MacLane chain. In the survey \cite{Kul} of F.V. Kulhmann, some VB examples are exhibited. Among them, we find  some cases where the inequality $\ti\bi\ge\ml(\phi)$ is an equality (Example 3.14), and some  where it is an strict inequality (Examples 3.12, 3.17, 3.20 and 3.22).\bs\bs

\end{document}